\documentclass{amsart}

\usepackage{amsthm}
\usepackage{amssymb}
\usepackage{amsmath}
\usepackage{amsmath,amscd}
\usepackage{epsfig}
\usepackage{graphicx}
\usepackage{url}

\newtheorem{theorem}{Theorem}[section]
\newtheorem{lemma}[theorem]{Lemma}
\newtheorem{claim}[theorem]{Claim}
\newtheorem{proposition}[theorem]{Proposition}
\newtheorem{corollary}[theorem]{Corollary}

\theoremstyle{definition}

\theoremstyle{remark}
\newtheorem{remark}[theorem]{Remark}

\numberwithin{equation}{section}

\newcommand{\RR}{{\mathbb R}}

\begin{document}

\title{An Immersed $S^2$ Self-Shrinker}

\author{Gregory Drugan}

\address{Department of Mathematics, University of Washington, Seattle, WA 98195}

\email{drugan@math.washington.edu}

\thanks{This work was partially supported by NSF RTG [DMS-0838212].}


\begin{abstract}
We construct an immersed and non-embedded $S^2$ self-shrinker.
\end{abstract}

\maketitle


\section{Introduction}

An immersion $F$ from a two-dimensional manifold $M$ into $\RR^3$ is a self-shrinker if it satisfies
\begin{equation}
\label{ss}
\Delta_g F = - \frac{1}{2}F^{\perp},
\end{equation}
where $g$ is the metric on $M$ induced by the immersion, $\Delta_g$ is the Laplace-Beltrami operator, and $F^{\perp}(p)$ is the projection of $F(p)$ into the normal space $N_p M$. When $F: M \to \RR^3$ satisfies~(\ref{ss}), the family of submanifolds $M_t = \sqrt{-t}F(M)$ is a solution of mean curvature flow for $t \in (-\infty, 0 )$. In the case where $M$ is compact, the rescalings $M_t$ shrink to the origin as $t$ approaches $0$ (hence the name self-shrinker). It is a consequence of Huisken's monotonicity formula~\cite{Hu} that a solution of mean curvature flow behaves asymptotically like a self-shrinker at a type I singularity. So, not only do self-shrinkers provide precious examples of solutions of mean curvature flow, but they also describe the behavior of mean curvature flow at certain singular points where the curvature blows-up. The simplest examples of self-shrinkers in $\RR^3$ are the sphere of radius 2 centered at the origin (the standard sphere), cylinders with an axis through the origin and radius $\sqrt{2}$, and planes through the origin. In this paper, we construct an immersed and non-embedded $S^2$ self-shrinker in $\RR^3$.
\begin{theorem}
\label{thm:1}
There exists an immersion $F: S^2 \to \RR^3$ satisfying $\Delta_g F = -\frac{1}{2} F^\perp$, and $F$ is not an embedding.
\end{theorem}

In 1989, Angenent~\cite{A} constructed an embedded self-shrinker with the topology type of a torus and provided numerical evidence for the existence of an immersed and non-embedded $S^2$ self-shrinker. (We note that the $S^2$ self-shrinker in Angenent's numerics is different from the one we construct.) In 1994, Chopp~\cite{Ch} described an algorithm for constructing surfaces that are approximately self-shrinkers and provided numerical evidence for the existence of a number of self-shrinkers, including compact, embedded self-shrinkers of genus 5 and 7. More recently, Kapouleas, Kleene, and M{\o}ller~\cite{KKM} and Nguyen~\cite{N1}--\cite{N3} used desingularization constructions to produce examples of complete, non-compact, embedded self-shrinkers with high genus in $\RR^3$. M{\o}ller~\cite{M} also used desingularization techniques to construct compact, embedded, high genus self-shrinkers in $\RR^3$. M{\o}ller's high genus examples, along with Angenent's torus and the standard sphere, are the only known examples of compact self-shrinkers in $\RR^3$. In contrast to these examples are several rigidity theorems for compact self-shrinkers. Huisken~\cite{Hu} showed that the only compact, mean-convex self-shrinker in $\RR^3$ is the standard sphere. In their study of generic singularities of mean curvature flow, Colding and Minicozzi~\cite{CM} showed that the only compact, embedded $F$-stable self-shrinker in $\RR^3$ is the standard sphere. As part of their classification of complete, embedded self-shrinkers with rotational symmetry, Kleene and M{\o}ller~\cite{KM} showed that the standard sphere is the only embedded $S^2$ self-shrinker with rotational symmetry. In an independent work~\cite{D}, we proved this result by showing that an embedded $S^2$ self-shrinker with rotational symmetry must be mean-convex. It is unknown whether or not the standard sphere is the only embedded $S^2$ self-shrinker in $\RR^3$.

The basic idea of the proof of Theorem~\ref{thm:1} is to construct a curve in the $(x,z)$-plane with self-intersections whose rotation about the $z$-axis is an $S^2$ self-shrinker. In this setting, the self-shrinker equation~(\ref{ss}) reduces to a differential equation. When the curve can be written in the form $(x, \gamma(x))$, the differential equation is
\begin{equation}
\label{gamma:ode}
\frac{\gamma''}{1+(\gamma')^2} = \left(\frac{1}{2}x - \frac{1}{x} \right) \gamma' -\frac{1}{2} \gamma.
\end{equation}
Using comparison arguments we describe the behavior of solutions of the differential equation for a range of initial conditions. Then, following the approach of Angenent in~\cite{A}, we use a continuity argument to find an initial condition that corresponds to a solution whose rotation about the $z$-axis is an immersed and non-embedded $S^2$ self-shrinker.

The curve we construct in the proof of Theorem~\ref{thm:1} (see Figure~\ref{fig_immersed}) is symmetric with respect to reflections across the $x$-axis, and it is enough to describe this curve as it travels from the positive $z$-axis to the point where it intersects the $x$-axis perpendicularly. We start the construction by studying solutions of~(\ref{gamma:ode}) with $\gamma(0) > 0$ and $\gamma'(0)=0$. Notice that one of the terms in the differential equation involves $\frac{1}{x}$, and hence this equation has a singularity at $x=0$. We begin Section~\ref{gamma:branch} by discussing the existence, uniqueness, and continuous dependence on initial height of solutions when $x$ is near $0$. As we move away from the origin, we can use existence theorems for differential equations to show that a solution $\gamma$ will exist until it blows-up. Next, we show that $\gamma$ is decreasing and concave down, and for small initial height, it must cross the $x$-axis before it blows-up. By a theorem of Lu Wang~\cite{W}, we know that $\gamma$ blows-up at a finite point $x_*$, and we use a comparison argument to estimate $\gamma'$ and show there is a finite point $z_*$ so that $\gamma(x_*) = z_*$. We finish Section~\ref{gamma:branch} by showing $x_* \to \infty$ and $z_* \to 0$ as the initial height approaches $0$.

\begin{figure}
\begin{center}
\includegraphics[width=3.5in]{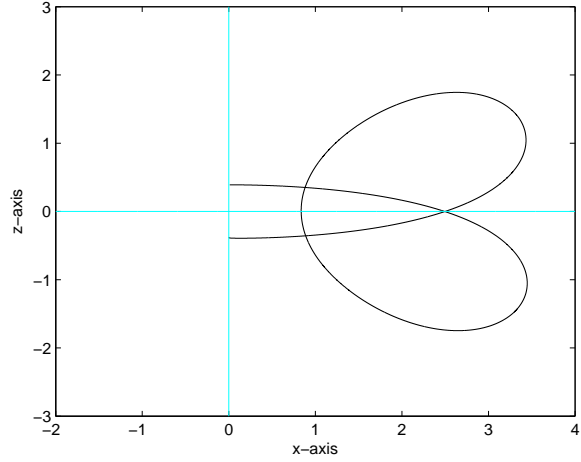}
\end{center}
\caption{\small A numerical approximation of a curve whose rotation about the $z$-axis is an immersed and non-embedded $S^2$ self-shrinker in $\RR^3$.}
\label{fig_immersed}
\end{figure}

In Section~\ref{cfsb}, we study the behavior of the curve $(x,\gamma(x))$ near the point $(x_*,z_*)$. Writing the curve $(x,\gamma(x))$ in the form $(\alpha(z),z)$, we get a solution of the differential equation
\begin{equation}
\label{alpha:eq}
\frac{\alpha''}{1+(\alpha')^2} = \left( \frac{1}{\alpha} - \frac{1}{2} \alpha \right) + \frac{1}{2} z \alpha'.
\end{equation}
At $(x_*,z_*)$, we have $\alpha(z_*) = x_*$ and $\alpha'(z_*) = 0$, and by the existence theory for differential equations, we can continue the curve $(x,\gamma(x))$ past the blow-up point $(x_*,z_*)$ along $(\alpha(z),z)$. Following $(\alpha(z),z)$, we show that the curve makes a turn at $(x_*,z_*)$ and heads back towards the $z$-axis. The curve heading back towards the $z$-axis can be written as $(x,\beta(x))$, where $\beta(x)$ is a solution of~(\ref{gamma:ode}).  Applying existence, uniqueness, and continuity theorems to these differential equations, we discuss how the curves $\gamma$ and $\alpha$ and the point $(x_*,z_*)$ depend continuously on the initial height. We also show that $z_* < 0$ when $\gamma(0) \in (0,2)$. Although this last result is not essential to the construction, it is a nice consequence of the rigidity of compact, mean-convex self-shrinkers due to Huisken~\cite{Hu}.

In Section~\ref{beta:branch}, we study the solutions $\beta(x)$ as they travel from $(x_*,z_*)$ toward the $z$-axis. We know there exists a point $x_{**} \geq 0$ so that $\beta$ is a solution of~(\ref{gamma:ode}) on $(x_{**},x_*)$ and either $\beta$ blows-up as $x$ approaches $x_{**}$ or $x_{**}=0$. We show for small initial height $\gamma(0)$ that $\beta$ achieves a negative minimum at a point $x_m \in (x_{**}, x_*)$, $\beta$ is concave up, $x_{**}>0$, and $0<\beta(x_{**}) < \infty$. To prove $\beta(x_{**})>0$, we give a direct argument, which shows how the singular term in~(\ref{gamma:ode}) forces $\beta$ to cross the $x$-axis when $x_{**}$ is small. This direct crossing argument is different from the limiting argument used by Angenent in~\cite{A}, and the analysis of $\beta$ in this section leads to a different construction of Angenent's torus self-shrinker. We also note that M{\o}ller~\cite{M} constructed a torus self-shrinker with explicit estimates on the cross-sections, which he used to construct the high genus compact, embeddded self-shrinkers.

Finally, in Section~\ref{proof} we finish the proof of the Theorem~\ref{thm:1}. We let $\gamma_b$ be the solution of~(\ref{gamma:ode}) with $\gamma_b(0)=b$, and define $\beta_b$, $x_*^b$, $x_{**}^b$, and $x_m^b$ as above. Following Angenent's argument in~\cite{A}, we consider the intial height $b_0$ given by $$b_0 = \sup \{ \tilde{b} \, : \, \forall b \in (0, \tilde{b}], \exists x_m^b \in (x_{**}^b, x_*^b) \textrm{ so that } \beta_b'(x_m^b) = 0 \textrm{ and } \beta_b(x_{**}^b) > 0 \}.$$ Using continuity arguments we show that $\beta_{b_0}$ intersects the $x$-axis perpendicularly at $x_{**}^{b_0}$. Thus, the curve $\gamma_{b_0} \cup \beta_{b_0} \cup -\beta_{b_0} \cup -\gamma_{b_0}$ is a smooth curve in the right-half plane that intersects the $z$-axis perpendicularly at precisely two points (see Figure~\ref{fig_immersed}), and the rotation of this curve about the $z$-axis is an immersed and non-emedded $S^2$ self-shrinker in $\RR^3$.

\begin{remark}
The proof works in higher dimensions to give an immersed and non-embedded $S^n$ self-shrinker in $\RR^{n+1}$. In this setting, the $\frac{1}{x}$ singular term in~(\ref{gamma:ode}) is replaced with $\frac{n-1}{x}$. In the one-dimensional case, the singular term vanishes and solutions can cross over the line $x=0$ without the slope restriction that holds in higher dimensions. The compact one-dimensional self-shrinkers were completely classified by both Abresch and Langer~\cite{AL} and Epstein and Weinstein~\cite{EW}. In this case, the standard circle is the only embedded $S^1$ self-shrinker, and there are many immersed and non-embedded $S^1$ self-shrinkers.
\end{remark}


\section{The First Branch}
\label{gamma:branch}

In this section we study solutions of~(\ref{gamma:ode}) with $\gamma(0) = b >0$ and $\gamma'(0) = 0$. We begin by discussing the existence, uniqueness, and continuous dependence on initial height of solutions. After this, we use a variety of comparison estimates to describe the basic behavior of $\gamma$ when $b > 0$. Finally, we finish the section with a detailed description of $\gamma$ when the initial height $b>0$ is small.

\subsection{Existence of Solutions Near $x=0$}
\label{gamma:ps}

Notice that one of the terms in~(\ref{gamma:ode}) involves $\frac{1}{x}$, and hence this equation has a singularity at $x=0$. We have the following proposition addressing the existence, uniqueness, and continuous dependence of solutions on initial height when $x$ is near $0$.

\begin{proposition}
\label{ps:prop}
For any $b \in \RR$, there exists $A=A(b)>0$ and a unique analytic function $\gamma$ defined on $[0,1/A]$ so that $\gamma(0) = b$, $\gamma'(0)=0$, and $\gamma$ is a solution of~\emph{(\ref{gamma:ode})}. Moreover, $\gamma$ and $\gamma'$ depend continuously on $b$ as follows: For each $M>0$ there exists $A=A(M)>0$ with the property that for $\varepsilon >0$ there is a $\delta >0$ so that if $|b_1 -b_2| < \delta$ and $|b_i| \leq M$, then $|\gamma_1(x) -\gamma_2(x)| < \varepsilon$ and $|\gamma_1'(x) - \gamma_2'(x)| < \varepsilon$ for all $x\in[0,1/A]$, where $\gamma_i$ is the unique analytic solution of~\emph{(\ref{gamma:ode})} with $\gamma_i(0)=b_i$.
\end{proposition}

\begin{proof}
We mention two proofs of the proposition. First, using a power series argument specific to the equation~(\ref{gamma:ode}) we established the existence, uniqueness, and continuity results as stated in the proposition. This argument is included in the appendix. Afterwards, Robin Graham informed us of the general reference~\cite{BG}, where the Cauchy problem for singular systems of partial differential equations was studied. Applying Theorem 2.2 from~\cite{BG} also shows that~(\ref{gamma:ode}) has a unique analytic solution in a neighborhood of $0$. We would like to thank Robin Graham for this reference.
\end{proof}


\subsection{Basic shape of $\gamma$}

Let $\gamma$ be the solution of~(\ref{gamma:ode}) with $\gamma(0) = b>0$ and $\gamma'(0)=0$. Then $\gamma''(0) = -b/4$ so that $\gamma$ starts out concave down. Taking derivatives of~(\ref{gamma:ode}), we have the following equations.
\begin{equation}
\label{gamma:eq:3}
\frac{\gamma'''}{1+(\gamma')^2} = \frac{2 \gamma' (\gamma'')^2}{( 1+(\gamma')^2 )^2} +  \left(\frac{1}{2}x - \frac{1}{x} \right) \gamma'' + \frac{1}{x^2} \gamma'
\end{equation}
and
\begin{eqnarray}
\label{gamma:eq:4}
\frac{\gamma^{(iv)}}{1+(\gamma')^2} & = & \frac{6 \gamma' \gamma'' \gamma''' + 2(\gamma'')^3}{( 1+(\gamma')^2 )^2} - \frac{8 (\gamma')^2 (\gamma'')^3}{( 1+(\gamma')^2 )^3}  + \left(\frac{1}{2}x - \frac{1}{x} \right) \gamma''' \\
& &  + \left(\frac{1}{2} +  \frac{2}{x^2} \right) \gamma'' - \frac{2}{x^3} \gamma' \nonumber
\end{eqnarray}

\begin{claim}
\label{claim:ccd}
$\gamma'' < 0$.
\end{claim}
\begin{proof}
Since $\gamma''(0)= -\frac{b}{4}$, we know that $\gamma''<0$ near 0. Suppose $\gamma''(x) = 0$ for some $x>0$. Choose $\bar{x}$ so that $\gamma''(\bar{x}) = 0$ and $\gamma''(x) < 0$ for $x \in [0,\bar{x})$. Then $\gamma'''(\bar{x}) \geq 0$. Also, $\gamma'(\bar{x}) < 0$ (since $\gamma'(0) = 0$). Using~(\ref{gamma:eq:3}), we see that $$0 \leq \frac{\gamma'''(\bar{x})}{1+\gamma'(\bar{x})^2} = \frac{1}{\bar{x}^2}\gamma'(\bar{x}) < 0,$$ which is a contradiction.
\end{proof}

In~\cite{W}, Lu Wang proved that an entire self-shrinker graph must be a plane. It follows that $\gamma$ cannot be defined on all of $[0,\infty)$, and therefore by the existence theory for differential equations there must be a point $x_* < \infty$ so that $\gamma$ blows-up at $x_*$ (blows-up in the sense that either $|\gamma|$ or $|\gamma'|$ goes to $\infty$ as $x$ goes to $x_*$). Since $\gamma''<0$ and $x_* < \infty$, it follows that $\lim_{x \to x_*} \gamma'(x) = -\infty$.

\begin{claim}
$x_* > \sqrt{2}$.
\end{claim}
\begin{proof}
Suppose $x_* < \sqrt{2}$. Then using~(\ref{gamma:ode}), we see that $\lim_{x \to x_*} \gamma''(x) = \infty$, which contradicts the fact that $\gamma'' < 0$. On the otherhand, the existence of the cylinder self-shrinker prevents $x_*$ from being equal to $\sqrt{2}$. To see this, suppose $x_* = \sqrt{2}$. Since $\gamma'' < 0$, we know that $\gamma(x) > 0$ for $x \in [0,x_*)$, and therefore there exists $z_* \geq 0$ so that $\lim_{x \to x_*}\gamma(x) = z_*$. Near the point $(\sqrt{2}, z_*)$, we write the curve $(x,\gamma(x))$ as $(\alpha(z),z)$, where $\alpha$ satisfies the differential equation~(\ref{alpha:eq}). Now, $\alpha(z_*) = \sqrt{2}$ and $\alpha'(z_*) = 0$, and by the uniqueness of solutions for this differential equation, $\alpha$ must be the constant function $\alpha(z) = \sqrt{2}$ (which corresponds to the cylinder self-shrinker). This contradicts the fact that $(x,\gamma(x))$ agrees with $(\alpha(z),z)$ near $(\sqrt{2}, z_*)$.
\end{proof}

\begin{lemma}
\label{bd:blowup}
$\lim_{x \to x_*} \gamma(x) > -\infty$.
\end{lemma}
\begin{proof}
Let $\gamma$ be a solution of~(\ref{gamma:ode}) with $\gamma(0) = b > 0$ and $\gamma'(0)=0$. Let $x_* > \sqrt{2}$ be the point where $\gamma$ blows-up. Fix $0< \delta < 1$ so that $x_* - \delta > \sqrt{2}$, and let $m>0$ be such that $(\frac{1}{2}x - \frac{1}{x}) \geq m$ when $x \in (x_* - \delta, x_*)$. Choose $M>0$ so that $m \geq \frac{3}{2M^2}$ and $M \geq -\gamma'(x_* - \delta)$.

For $\varepsilon>0$, define $g_\varepsilon(x)$ on $(x_* - \delta, x_* - \varepsilon)$ by $$g_\varepsilon(x) = \frac{M}{\sqrt{(x_* - \varepsilon) - x}}.$$ Then $$g_\varepsilon ''(x) = \frac{3}{2M^2} g_\varepsilon(x)^2 g_\varepsilon'(x) \leq \left( \frac{1}{2}x - \frac{1}{x} \right) g_\varepsilon(x)^2 g_\varepsilon'(x) ,$$ for $x \in (x_* - \delta, x_* - \varepsilon)$. We will use the function $g_\varepsilon$ to show that $-\gamma'$ blows-up no faster than $M/\sqrt{x_*-x}$. Let $f(x) = -\gamma'(x)$. Then $f \geq 0$, $f' > 0$, and by~(\ref{gamma:eq:3}), $$f''(x) \geq \left(\frac{1}{2}x - \frac{1}{x} \right)f(x)^2 f'(x),$$  when $x \geq \sqrt{2}$. We will show $f \leq g_\varepsilon$. By construction, $$f(x_*-\delta) \leq M < g_\varepsilon(x_*-\delta)$$ and $$f(x_*-\varepsilon) < \lim_{x \to (x_* - \varepsilon)} g_\varepsilon(x).$$ Therefore, if $f>g_\varepsilon$ at some point, then $f-g_\varepsilon$ achieves a positive maximum at some point $x' \in (x_* - \delta, x_* - \varepsilon)$. This leads to $(f-g_\varepsilon)'(x') = 0$ and $(f-g_\varepsilon)''(x') \leq 0$. Consequently, $$0 \geq (f-g_\varepsilon)''(x') \geq \left(\frac{1}{2}x' - \frac{1}{x'} \right)f'(x') \left( f(x')^2-g_\varepsilon(x')^2 \right) > 0,$$ which is a contradiction.

It follows that $f \leq g_\varepsilon$ on $(x_* - \delta, x_* - \varepsilon)$. Taking $\varepsilon \to 0$, we conclude that $$\gamma'(x) \geq \frac{-M}{\sqrt{x_* -x}},$$ for $x \in (x_* - \delta, x_*)$. Therefore, $\lim_{x \to x_*}\gamma(x) > -\infty$.
\end{proof}

\begin{remark}
At this point, we can give a basic description of the $\gamma$ curves: For $b>0$, let $\gamma_b$ denote the solution of~(\ref{gamma:ode}) with $\gamma_b(0) = b$ and $\gamma_b'(0) = 0$. Then $\gamma_b$ is decreasing and concave down, and there exists a point $x_*^b \in (\sqrt{2} , \infty)$ so that $\gamma_b$ is defined on $[0,x_*)$ and $\lim_{x \to x_*^b} \gamma_b'(x) = -\infty$. There also exists a point $z_*^b \in (-\infty, b)$ so that $\gamma_b(x_*^b) = z_*^b$.
\end{remark}

\subsection{Estimates for small initial height}

In this section, we prove estimates for $x_*^b$ and $z_*^b$ when the initial height $b>0$ is small.
\begin{proposition}
\label{gamma:small:height}
For $b>0$, let $\gamma_b$ denote the solution of~\emph{(\ref{gamma:ode})} with $\gamma_b(0) = b$ and $\gamma_b'(0) = 0$. Let $x^{b}_*$ denote the point where $\gamma_{b}$ blows-up, and let $z^b_* = \gamma_b(x^b_*)$.

There exists $\bar{b} > 0$ so that if $b \in (0,\bar{b}]$, then $$x_*^b \geq \sqrt{ \log{\frac{2}{\pi b^2}} },$$ $$\frac{-12}{ \sqrt{\log{\frac{2}{\pi b^2}}} } \leq z_*^b < 0,$$ and there exists a point $x_0^b \in [2, 2\sqrt{2}]$ so that $\gamma_b(x_0^b) = 0$.
\end{proposition}

Before we prove the proposition, we prove some results about solutions of~(\ref{gamma:ode}) when the initial height is small. Let $\gamma$ be the solution of~(\ref{gamma:ode}) with $\gamma(0) = b$ and $\gamma'(0) = 0$.

\begin{claim}
\label{c3:fb}
If $b < \sqrt{ \frac{2}{3 \pi} } \cdot \frac{1}{e^4}$, then $x_* > 2 \sqrt{2}$ and $|\gamma'(x)| \leq \frac{\sqrt{3}}{3}$ for $x \in [0,2\sqrt{2}]$.
\end{claim}
\begin{proof}
Since $\gamma'(0) = 0$, $\gamma''<0$, and $\lim_{x \to x_*} \gamma'(x) = - \infty$, we know there exists $x' \in (0,x_*)$ so that $\gamma'(x') = -\frac{\sqrt{3}}{3}$. For $x \in (0,x')$, we have
\begin{eqnarray}
\frac{d}{dx} \left( e^{-\frac{x^2}{2}} \gamma'(x) \right) & = &  e^{-\frac{x^2}{2}} \gamma''(x) - x e^{-\frac{x^2}{2}} \gamma'(x)  \nonumber \\
& \geq &  \frac{2}{1 + \gamma'(x)^2} e^{-\frac{x^2}{2}} \gamma''(x) - x e^{-\frac{x^2}{2}} \gamma'(x)  \nonumber \\
& = &  2 e^{-\frac{x^2}{2}} \left[ \left( \frac{1}{2}x - \frac{1}{x} \right) \gamma' -\frac{1}{2} \gamma \right] - x e^{-\frac{x^2}{2}} \gamma'(x)  \nonumber \\
& = &  -2 e^{-\frac{x^2}{2}} \frac{1}{x} \gamma' - e^{-\frac{x^2}{2}} \gamma(x)  \nonumber \\
& \geq & - e^{-\frac{x^2}{2}} \gamma(x). \nonumber
\end{eqnarray}
Integrating from $0$ to $x'$, $$-\frac{\sqrt{3}}{3} e^{-\frac{(x')^2}{2}} \geq - \int_0^{x'} e^{-\frac{x^2}{2}} \gamma(x) dx \geq - b \int_0^{x'} e^{-\frac{x^2}{2}} dx\geq -b \sqrt{\frac{\pi}{2}}.$$ When $b < \sqrt{ \frac{2}{3 \pi} } \cdot \frac{1}{e^4}$ we have $e^{-\frac{(x')^2}{2}} < e^{-4}$, and therefore $x' > 2\sqrt{2}$.
\end{proof}

\begin{claim}
\label{c1:fb}
If $|\gamma'(x)| \leq \frac{\sqrt{3}}{3}$ for $x \in [0,\sqrt{2}]$, then $\gamma'''(x) < 0$ for $x \in (0, x_*)$.
\end{claim}
\begin{proof}
From the power series expansion for $\gamma$ at $x=0$, we know that $\gamma'''(0) = 0$ and $\gamma^{(iv)}(0) < 0$. Therefore, $\gamma'''(x) < 0$ when $x > 0$ is near 0. Also, using~(\ref{gamma:eq:3}), we see that $\gamma'''(x) < 0$ when $x \geq \sqrt{2}$. Suppose $\gamma'''(x) = 0$ for some $x>0$. Then there exists $\bar{x} \in (0,\sqrt{2})$ so that $\gamma'''(\bar{x}) = 0$ and $\gamma'''(x) < 0$ for $x \in (0,\bar{x})$. It follows that $\gamma^{(iv)}(\bar{x}) \geq 0$. Notice that $x \gamma''(x) -\gamma'(x)$ is decreasing and hence negative on $(0,\bar{x})$. Then, using~(\ref{gamma:eq:4}) and the assumption that $|\gamma'(\bar{x})| \leq \frac{\sqrt{3}}{3}$, we see that $$\frac{\gamma^{(iv)}(\bar{x})}{1+\gamma'(\bar{x})^2} = 2 (\gamma''(\bar{x}))^3 \frac{1 - 3 (\gamma'(\bar{x}))^2}{(1+\gamma'(\bar{x})^2)^3} + \frac{1}{2} \gamma''(\bar{x}) + 2 \frac{\bar{x} \gamma''(\bar{x}) - \gamma'(\bar{x})}{(\bar{x})^3} < 0,$$ which is a contradiction.
\end{proof}

\begin{claim}
\label{c2:fb}
If $|\gamma'(x)| \leq \frac{\sqrt{3}}{3}$ for $x \in [0,2\sqrt{2}]$ and $b < \frac{1}{2}$, then $\frac{x \gamma'(x) - \gamma(x)}{\sqrt{1 + \gamma'(x)^2}}$ is non-increasing on $[0,2\sqrt{2}]$.
\end{claim}
\begin{proof}
Looking at the derivative of $\frac{x \gamma'(x) - \gamma(x)}{\sqrt{1 + \gamma'(x)^2}}$: $$ \frac{d}{dx} \left( \frac{x \gamma'(x) - \gamma(x)}{\sqrt{1 + \gamma'(x)^2}} \right) = \gamma''(x) \frac{x + \gamma(x) \gamma'(x)}{(1 + \gamma'(x)^2)^{3/2}},$$ we see that it is enough to show $x + \gamma(x) \gamma'(x) \geq 0$. Since $x + \gamma(x) \gamma'(x)$ equals 0 when $x=0$, it is sufficient to show $1 + \gamma(x) \gamma''(x) + \gamma'(x)^2 \geq 0$ on $(0,2\sqrt{2}]$.  For $x \in (0,2\sqrt{2}]$, assuming $|\gamma'(x)| \leq \frac{\sqrt{3}}{3}$ and $b<\frac{1}{2}$, we have
\begin{eqnarray}
\gamma''(x) & = & (1+\gamma'(x)^2) \left[ \left(\frac{1}{2}x - \frac{1}{x} \right) \gamma'(x) - \frac{1}{2} \gamma(x) \right] \nonumber \\
& \geq & \frac{4}{3} \left[ - \sqrt{2} \frac{\sqrt{3}}{3} - \frac{b}{2} \right] \geq -2, \nonumber
\end{eqnarray}
and it follows that $1 + \gamma(x) \gamma''(x) + \gamma'(x)^2 \geq 0$.
\end{proof}

\begin{lemma}
\label{c4:fb}
Suppose $x_* > 2\sqrt{2}$ and there exists a point $x_0 \in [2,2\sqrt{2}]$ so that $\gamma(x_0) = 0$. Then, for $x \in [x_0, x_*)$, $$\gamma(x) > \frac{8}{x}\gamma'(x).$$
\end{lemma}
\begin{proof}
Let $\Phi(x) = \frac{1}{8} x \gamma(x) - \gamma'(x)$. We want to show $\Phi(x) > 0$. We know that $\Phi(x_0) = -\gamma'(x_0) > 0$. We also have
\begin{eqnarray}
\frac{1}{8}x \gamma(x) & = & \frac{1}{8} x \int_{x_0}^x \gamma'(\xi) d\xi \nonumber \\
& > & \frac{1}{8} x (x-x_0) \gamma'(x). \nonumber
\end{eqnarray}
Since $x_0 \geq 2$, we see that $\Phi(x) > 0$ when $x \leq 4$.

Suppose $\Phi(x) = 0$ for some $x \in [x_0, x_*)$. Then $x >4$ and there exists a point $\bar{x} \in (4,x_*)$ so that $\Phi(\bar{x}) = 0$ and $\Phi(x) > 0$ for $x \in [x_0, \bar{x})$. This implies that $\Phi'(\bar{x}) \leq 0$ and $\frac{1}{8}\bar{x} \gamma(\bar{x}) = \gamma'(\bar{x})$. Since $\bar{x} > 4$ and $\gamma(\bar{x}) < 0$, we have
\begin{eqnarray}
\Phi'(\bar{x}) & = & \frac{1}{8} \gamma(\bar{x}) + \frac{1}{8} \bar{x} \gamma'(\bar{x}) - \gamma''(\bar{x}) \nonumber \\
& \geq & \frac{1}{8} \gamma(\bar{x}) + \frac{1}{8} \bar{x} \gamma'(\bar{x}) - \frac{\gamma''(\bar{x})}{1+\gamma'(\bar{x})^2} \nonumber \\
& = & \frac{1}{8} \gamma(\bar{x}) + \frac{1}{8} \bar{x} \gamma'(\bar{x}) - \left[ \left(\frac{1}{2}\bar{x} - \frac{1}{\bar{x}} \right) \gamma'(\bar{x}) - \frac{1}{2} \gamma(\bar{x}) \right]  \nonumber \\
& = & \gamma(\bar{x}) \left( \frac{1}{8} + \frac{1}{64} (\bar{x})^2 - \left[ \left(\frac{1}{2}\bar{x} - \frac{1}{\bar{x}} \right) \frac{1}{8} \bar{x} - \frac{1}{2} \right] \right) \nonumber \\
& = & \gamma(\bar{x}) \left( \frac{3}{4} - \frac{3}{64}(\bar{x})^2 \right) > 0, \nonumber
\end{eqnarray}
which is a contradiction.
\end{proof}

\begin{proof}[Proof of Proposition~\ref{gamma:small:height}]
Let $b>0$, and let $\gamma$ be the solution of~(\ref{gamma:ode}) with $\gamma(0) = b$ and $\gamma'(0) = 0$. We know there exists a point $x_* \in (\sqrt{2}, \infty)$ and a point $z_* \in (-\infty, b)$ so that $\lim_{x \to x_*} \gamma'(x) = - \infty$ and $\gamma(x_*) = z_*$. We assume $b < \sqrt{ \frac{2}{3 \pi} } \cdot \frac{1}{e^4}$ (and also $b<\frac{1}{2}$). By Claim~\ref{c3:fb}, we know that $x_* > 2\sqrt{2}$ and $|\gamma'(x)| \leq \frac{\sqrt{3}}{3}$ for $x \in [0,2\sqrt{2}]$. Then by Claim~\ref{c1:fb} we have $\gamma'''<0$ on $(0,x_*)$. Integrating this inequality from $0$ to $x$ repeatedly, we see that $$\gamma(x) < b(1- \frac{1}{8}x^2).$$ Since $x_*>2\sqrt{2}$, it follows that there exists $x_0 \in (0,2\sqrt{2})$ so that $\gamma(x_0) = 0$.

To estimate $x_0$ from below, we write equation~(\ref{gamma:ode}) in the form
\begin{equation}
\label{gamma:div:form}
\frac{d}{dx} \left( \frac{x \gamma'(x)}{ \sqrt{1+\gamma'(x)^2} } \right) = \frac{1}{2} x \cdot \frac{x \gamma'(x) - \gamma(x) }{ \sqrt{1+\gamma'(x)^2} }.
\end{equation}
It follows from Claim~\ref{c2:fb} that $\frac{x \gamma'(x) - \gamma(x) }{ \sqrt{1+\gamma'(x)^2} } \geq \frac{x_0 \gamma'(x_0) }{ \sqrt{1+\gamma'(x_0)^2} }$, for $x \in [0,x_0]$, and thus by integrating~(\ref{gamma:div:form}) from $0$ to $x_0$, we get
\begin{eqnarray}
\frac{x_0 \gamma'(x_0)}{ \sqrt{1+\gamma'(x_0)^2} } & = & \int_0^{x_0} \frac{1}{2} x \cdot \frac{x \gamma'(x) - \gamma(x) }{ \sqrt{1+\gamma'(x)^2} } dx \nonumber \\
& \geq & \frac{x_0 \gamma'(x_0)}{ \sqrt{1+\gamma'(x_0)^2} } \int_0^{x_0} \frac{1}{2} x dx. \nonumber 
\end{eqnarray}
Therefore, $1 \leq \frac{(x_0)^2}{4}$ and we see that $x_0 \geq 2$. This proves the last statement in the proposition.

Next, we want to slightly refine the estimate from Claim~\ref{c3:fb} to establish a lower bound for $x_*$ in terms of $b$. This will simplify the constants that appear in the following calculations. Let $x_1 \in (0,x_*)$ be the point where $\gamma'(x_1) = -1$. Using the same argument we used in the proof of Claim~\ref{c3:fb}, we integrate the inequality $$\frac{d}{dx} \left( e^{-\frac{x^2}{2}} \gamma'(x) \right) \geq - e^{-\frac{x^2}{2}} \gamma(x)$$ from $0$ to $x_1$ to conclude that $-e^{- \frac{(x_1)^2}{2}} \geq - b \sqrt{ \frac{\pi}{2} } $ and therefore, $$x_1 \geq \sqrt{ \log{ \frac{2}{\pi b^2} } }.$$ Since $x_0 \in [2, 2\sqrt{2}]$, it follows from Lemma~\ref{c4:fb} that $\gamma(x) > \frac{8}{x}\gamma'(x)$. In particular, at $x_1$, we have $$\gamma(x_1) > - \frac{8}{x_1} \geq - \frac{8}{ \sqrt{\log{\frac{2}{\pi b^2}}} }.$$ We will extend this estimate for $\gamma(x_1)$ to an estimate for $\gamma(x_*) = z_*$. We assume $b \leq \sqrt{\frac{2}{\pi e^{25}}}$ so that $x_1 \geq 5$. For $x \geq x_1$, we have
\begin{eqnarray}
\gamma''(x) & \leq & \gamma'(x)^2 \frac{\gamma''(x)}{1 + \gamma'(x)^2} \nonumber \\
& = & \gamma'(x)^2 \left[ \left( \frac{1}{2}x - \frac{1}{x} \right) \gamma'(x) - \frac{1}{2} \gamma(x) \right] \nonumber \\
& < & \gamma'(x)^2 \left( \frac{1}{2}x - \frac{5}{x} \right) \gamma'(x) \nonumber \\
& \leq &  \frac{1}{4} x \gamma'(x)^3, \nonumber
\end{eqnarray}
where we have used that $x \geq 5$ and $\gamma(x) > \frac{8}{x}\gamma'(x)$.

Integrating the previous inequality from $x$ to $x_*$, implies $$\gamma'(x)^2 \leq \frac{4}{(x_*)^2 - x^2},$$ for $x \geq x_1$. Since $\gamma'(x) < 0$, we have 
\begin{equation}
\label{gamma:slope:end}
\gamma'(x) \geq - \frac{2}{\sqrt{(x_*)^2 - x^2}} \geq - \frac{1}{\sqrt{x_* +x_1}} \cdot \frac{2}{\sqrt{x_* - x}},
\end{equation}
for $x \in [x_1, x_*)$. At $x_1$, this tells us that $$- \frac{\sqrt{x_* - x_1}}{\sqrt{x_* + x_1}} \geq - \frac{2}{x_* + x_1}.$$ Finally, integrating~(\ref{gamma:slope:end}) from $x_1$ to $x_*$, we have $$\gamma(x_*) - \gamma(x_1) \geq -\frac{4}{\sqrt{x_* + x_1}} \cdot \sqrt{x_* - x_1},$$ and therefore
\begin{eqnarray}
\gamma(x_*) & \geq & \gamma(x_1) - \frac{4}{\sqrt{x_* + x_1}} \cdot \sqrt{x_* - x_1} \nonumber \\
& \geq & - \frac{8}{x_1} -  \frac{8}{x_* + x_1} \nonumber \\
& \geq & -\frac{12}{x_1} \, \geq \, - \frac{12}{ \sqrt{\log{\frac{2}{\pi b^2}}} }, \nonumber
\end{eqnarray}
which completes the proof of the proposition with $\bar{b} = \sqrt{\frac{2}{\pi e^{25}}}$.
\end{proof}


\section{Connecting the First and Second Branches}
\label{cfsb}

Given the basic shape of $\gamma$ described in the previous section, we know that $\gamma'(x) < 0$ when $x>0$, and thus, for $x>0$, the curve $(x,\gamma(x))$ can be written as $(\alpha(z),z)$. Since $\alpha'(z) = 1 / \gamma'(\alpha(z))$ and $\gamma$ is a solution of~(\ref{gamma:ode}), it follows that $\alpha$ is a solution of~(\ref{alpha:eq}) with $\alpha(z_*) = x_*$ and $\alpha'(z_*) = 0$. In particualr, $\alpha''(z_*) = \frac{1}{x_*} - \frac{1}{2}x_* < 0$. This shows us that the $\alpha$ curve is concave down at $(x_*,z_*)$ and heads back towards the $z$-axis as $z$ decreases.  More precisely, in a neighborhood of $z_*$, we have $\alpha'(z)>0$ when $z<z_*$. It follows that the curve $(\alpha(z),z)$ can be written as a curve $(x,\beta(x))$ where $\beta(x)$ satisfies~(\ref{gamma:ode}). Using the existence theory for differential equations, we know that there exists $x_{**} < x_*$ so that $\beta$ is a solution of~(\ref{gamma:ode}) on $(x_{**},x_*)$. Here $x_{**} \geq 0$ is chosen so that $\beta$ blows-up as $x$ approaches $x_{**}$ or $x_{**}=0$. We note that $\beta(x_*) = z_*$ and $\lim_{x \to x_*} \beta'(x) = \infty$, and these two conditions uniquely determine $\beta$ as a solution of~(\ref{gamma:ode}).

In the next section we will study the behavior of $\beta$. For the remainder of this section, we discuss how the curves $\gamma$ and $\alpha$ and the point $(x_*,z_*)$ depend continuously on the initial height. From Proposition~\ref{ps:prop} we know that the $\gamma$ curves depend continuously on the initial height in a neighbrhood of 0. Once we move away from the singularity at 0, if we rewrite~(\ref{gamma:ode}) as a first order system, then  a direct application of the existence, uniqueness, and continuity theorems for differential equations extends this continuity:

\begin{proposition}
\label{prop:gamma:cont}
For $b>0$, let $\gamma_b$ denote the unique solution of~\emph{(\ref{gamma:ode})} with $\gamma_b(0) = b$ and $\gamma_b'(0) = 0$. Let $x^{b}_*$ denote the point where $\gamma_{b}$ blows-up, and let $z^b_* = \gamma_b(x^b_*)$.

Fix $b_0>0$. Then, for $\varepsilon > 0$, there exists $\delta>0$ so that if $|b-b_0| < \delta$ and $b>0$, then $\gamma_b(x)$ is defined on $[0, x^{b_0}_* - \varepsilon]$. Moreover, $$|\gamma_b(x) - \gamma_{b_0}(x)| + |\gamma_b'(x) - \gamma_{b_0}'(x)| < \varepsilon,$$ for $x \in [0, x^{b_0}_* - \varepsilon]$.
\end{proposition}

We end this section with a proposition which shows how the $\alpha$ curves depend continuously on the initial height. Again, the proof of the proposition is an application of the existence, uniqueness, and continuity theorems for differential equations.

\begin{proposition}
\label{prop:alpha:cont}
Fix $b_0>0$, and let $\alpha_{b_0}$ be the unique solution of~\emph{(\ref{alpha:eq})} with $\alpha_{b_0}(z_*^{b_0}) = x_*^{b_0}$ and $\alpha_{b_0}'(z_*^{b_0}) = 0$. Let $\rho > 0$ be chosen so that $[z_*^{b_0} - \rho, z_*^{b_0} + \rho]$ is contained in the maximal interval of existence for $\alpha_{b_0}$. Then, for $\varepsilon >0$, there exists $\delta >0$ so that if $|b-b_0| < \delta$ and $b>0$, then the unique solution $\alpha_b$ of~\emph{(\ref{alpha:eq})} with $\alpha_b(z_*^{b}) = x_*^{b}$ and $\alpha_b'(z_*^{b}) = 0$ is defined on $[z_*^{b_0} - \rho , z_*^{b_0} + \rho]$. Moreover, for $z \in [z_*^{b_0} - \rho,  z_*^{b_0} + \rho]$, $$|\alpha_b(z) - \alpha_{b_0}(z)| + |\alpha_b'(z) - \alpha_{b_0}'(z)| < \varepsilon.$$ 
\end{proposition}

As an application of Proposition~\ref{gamma:small:height} and Proposition~\ref{prop:alpha:cont} we use the rigidity of compact, mean-convex self-shrinkers due to Huisken~\cite{Hu} to show that $z_* < 0$ when $\gamma(0) \in (0,2)$. We note that this result is not essential to the proof of Theorem~\ref{thm:1}.

\begin{corollary}
\label{gamma:blowup:pt}
Fix $b_0>0$. Then, for $\varepsilon > 0$, there exists $\delta>0$ so that if $|b-b_0| < \delta$ and $b > 0$, then $|x^b_* - x^{b_0}_*| < \varepsilon$ and $|z^b_* - z^{b_0}_*| < \varepsilon$. Furthermore, if $b \in (0,2)$, then $z_*^b <0$.
\end{corollary}

\begin{proof}
The first statement is a consequence of Proposition~\ref{prop:alpha:cont}. To prove the second statement, let $$b_0 = \max \{ b' \in (0,2] \, : \, z_*^b<0 \textrm{ for } b \in (0, b'] \}.$$ It follows from Proposition~\ref{gamma:small:height} that $b_0$ is well-defined. We will show that $b_0 =2$. By definition of $b_0$, there exists an increasing sequence $b_n$ converging to $b_0$ so that $z_*^{b_n} < 0$. Applying the first part of the corollary, we have $z_*^{b_0} \leq 0$. If $z_*^{b_0} = 0$, then the rotation of the curve $\gamma_{b_0} \cup -\gamma_{b_0}$ about the $z$-axis is a convex, compact self-shrinker. By Huisken's classification of compact, mean-convex self-shrinkers, this must be the sphere of radius 2 centered at the origin, and in this case $b_0=2$. If $b_0 < 2$, then $z_*^{b_0} < 0$, and by the first part of the corollary, there exists $\delta>0$ so that $z_*^b < 0$ for $|b - b_0| < \delta$. This contradicts the choice of $b_0$ as the maximum of $\{ b' \in (0,2] \, : \, z_*^b<0 \textrm{ for } b \in (0, b'] \}$, and we conclude that $b_0 =2$. In particular, $z_*^b <0$ when $b \in (0,2)$.
\end{proof}


\section{The Second Branch}
\label{beta:branch}

In this section we study the $\beta$ curves as they travel from $(x_*,z_*)$ toward the $z$-axis. For $b>0$, let $\gamma$ denote the solution of~(\ref{gamma:ode}) with $\gamma(0) = b$ and $\gamma'(0) = 0$, let $x_*$ denote the point where $\gamma$ blows-up, and let $z_* = \gamma(x_*)$. Also, let $\beta$ denote the unique solution of~(\ref{gamma:ode}) with $\beta(x_*) = z_*$ and $\lim_{x \to x_*} \beta'(x) = \infty$, and let $x_{**} \in [0,x_*)$ be the point where $\beta$ blows-up, or if no such point exists, set $x_{**}=0$. We will show for small $b>0$ that $\beta$ achieves a negative minimum at a point $x_m \in (x_{**}, x_*)$, $\beta$ is concave up, $x_{**}>0$, and $0<\beta(x_{**}) < \infty$.

\subsection{Basic shape of $\beta$}

First, we prove some basic properties of the $\beta$ curves that are consequences of equations~(\ref{gamma:ode}) and~(\ref{gamma:eq:3}) and the fact that $\lim_{x \to x_*} \beta'(x) = \infty$.
\begin{claim}
\label{beta:claim1}
There exists at most one point $x_m \in (x_{**},x_*)$ for which $\beta'(x_m) = 0$. If such a point exists, then $\beta'' > 0$ on $(x_{**}, x_*)$.
\end{claim}
\begin{proof}
Using equation~(\ref{gamma:eq:3}), we see that $\beta'$ cannot vanish at more than one point. Now, suppose $\beta'(x_m) = 0$ for some $x_m \in (x_{**},x_*)$. By uniqueness of solutions, $\beta(x_m) \neq 0$ (otherwise $\beta$ would be identically 0), and thus $\beta''(x_m) \neq 0$. Since $\lim_{x \to x_*} \beta'(x) = \infty$, we know that $\beta'(x) > 0$ for $x < x_*$ and near $x_*$, and it follows that $\beta''(x_m) > 0$. Arguing as in the proof of Claim~\ref{claim:ccd}, we see that $\beta'' > 0$ on $(x_m,x_*)$ and similarly on $(x_{**},x_m)$. We note that $\beta(x_m) < 0$.
\end{proof}

\begin{claim}
\label{beta:zero}
If there exists $x_m \in (x_{**},x_*)$ so that $\beta'(x_m) = 0$, then $\beta(x)<0$ whenever $x \in (x_{**},x_m)$ and $x \geq \sqrt{2}$.
\end{claim}
\begin{proof}
In this case, we have $\beta'' > 0$ on $(x_{**},x_*)$. If $\beta(x) = 0$ for some $x \in (x_{**}, x_m)$, then $\beta'(x)<0$, and using~(\ref{gamma:ode}), we see that $x< \sqrt{2}$.
\end{proof}

\begin{claim}
$x_{**} < \sqrt{2}$.
\end{claim}
\begin{proof}
We treat the two cases from Claim~\ref{beta:claim1}. In the first case, there exists a point $x_m \in (x_{**},x_*)$ so that $\beta'(x_m)=0$ and $\beta'' > 0$ on $(x_{**},x_*)$. By Claim~\ref{beta:zero}, we know that $\beta(x) <0$ when $x \in (x_{**},x_m)$ and $x \geq \sqrt{2}$. Also, if we let $M = -\beta(x_m)$, then $\beta(x) \geq - M$ when $x \in (x_{**},x_*)$. Now, for any $\varepsilon > 0$, there exists a constant $m_\varepsilon >0$ so that $\frac{1}{2}x -\frac{1}{x} > m_\varepsilon$ for $x \geq \sqrt{2} + \varepsilon$, and using~(\ref{gamma:ode}), we have $\beta'(x) > -\frac{M}{2m_\varepsilon}$ when $x \geq \sqrt{2} + \varepsilon$. Therefore, $|\beta(x)|$ and $|\beta'(x)|$ are uniformly bounded for $x \geq \sqrt{2} + \varepsilon$ and away from $x_*$.  By the existence theory for differential equations $x_{**} < \sqrt{2} + \varepsilon$. Taking $\varepsilon \to 0$, we have $x_{**} \leq \sqrt{2}$. To see that $x_{**} < \sqrt{2}$, suppose $x_{**} = \sqrt{2}$. Then there exists $z_{**} \in [-M, 0]$ so that $\lim_{x \to x_{**}} \beta(x) = z_{**}$. It follows that near the point $(x_{**}, z_{**})$ the curve $(x,\beta(x))$ can be written as $(\bar{\alpha}(z),z)$ where $\bar{\alpha}$ is a solution of~(\ref{alpha:eq}) with $\bar{\alpha}(z_{**}) = \sqrt{2}$ and  $\bar{\alpha}'(z_{**}) = 0$. By the uniqueness of solutions of~(\ref{alpha:eq}) we deduce that $\bar{\alpha}$ is the constant function $\bar{\alpha}(z) = \sqrt{2}$, which is a contradiction.

In the second case, $\beta'>0$ on $(x_{**},x_*)$. If $x_{**} = 0$, we are done. Otherwise, $x_{**} > 0$ and consequently $\lim_{x \to x_{**}} \beta'(x) = \infty$. In this case, $\beta''$ must be negative at some point, and arguing as in the proof of Claim~\ref{claim:ccd}, we see that $\beta'' <0$ near $x_{**}$. It follows from~(\ref{gamma:ode}) that $x_{**} < \sqrt{2}$ when $\beta < 0$ near $x_{**}$. On the other hand, if $\beta \geq 0$, then $|\beta|$ is uniformly bounded (since $\beta \leq z_*$) and arguing as we did in the first case, we see that $x_{**} < \sqrt{2}$.
\end{proof}

\begin{lemma}
\label{beta:bd}
$\lim_{x \to x_{**}} \beta(x) < \infty$.
\end{lemma}
\begin{proof}
Suppose $\lim_{x \to x_{**}} \beta(x)  = \infty$. Since $\beta'>0$ near $x_*$, there exists a point $x_m \in (x_{**} ,x_* )$ so that $\beta'(x_m) = 0$. By Claim~\ref{beta:claim1}, we know that $\beta'' > 0$ and $\beta(x_m) < 0$. It follows that there exists a point $x_{\ell} \in (x_{**} ,x_m)$ so that $\beta(x_{\ell}) = 0$ with $\beta'(x_{\ell}) = -m < 0$. By Claim~\ref{beta:zero}, we know there exists $\delta>0$ so that $x_{\ell} + \delta < \sqrt{2}$, and in particular $\left(\frac{1}{2}x - \frac{1}{x} \right) \leq - \frac{1}{M}$ for some $M>0$, when $x \leq x_{\ell}$.

Let $f=-\beta'$. Then $f > 0$ when $x \in (x_{**}, x_{\ell})$ and $f'<0$. Using~(\ref{gamma:eq:3}), we have $$f'' \geq -\frac{1}{M}f'\cdot f^2,$$ when $x \in (x_{**}, x_{\ell})$. Fix $\varepsilon>0$, and let $$g_\varepsilon(x)= \frac{ m\sqrt{x_{\ell} - (x_{**} + \varepsilon) } + \sqrt{3M} }{\sqrt{x - (x_{**} + \varepsilon) }}.$$ Then $$g_\varepsilon'' = -\frac{3}{2} \frac{1}{(m\sqrt{x_{\ell} - (x_{**} + \varepsilon) } + \sqrt{3M})^2}g_\varepsilon'\cdot g_\varepsilon^2 \leq -\frac{1}{M}g_\varepsilon'\cdot g_\varepsilon^2,$$ for $x \in (x_{**}+ \varepsilon, x_{\ell})$.

Now, $g_\varepsilon(x_{\ell}) > f(x_{\ell})$ and $g_\varepsilon(x_{**}+\varepsilon) > f(x_{**}+\varepsilon)$. Suppose $f>g_\varepsilon$ at some point in $(x_{**}+\varepsilon, x_{\ell})$, then $f-g_\varepsilon$ must achieve a positive maximum in $(x_{**}+\varepsilon, x_{\ell})$. At such a point $$0 \geq (f-g_\varepsilon)'' \geq -\frac{1}{M}f'(f^2-g^2) > 0,$$ which is a contradiction. Therefore, $g_\varepsilon \geq f$. Taking $\varepsilon \to 0$, we have the estimate $$f(x) \leq \frac{ m\sqrt{x_{\ell} - x_{**} } + \sqrt{3M} }{\sqrt{x - x_{**} }} ,$$ for $x \in (x_{**}, x_{\ell})$. Integrating from $x$ to $x_{\ell}$, $$\beta(x) - \beta(x_{\ell}) \leq 2 \left( m\sqrt{x_{\ell} - x_{**} } + \sqrt{3M}  \right) \left( \sqrt{x_{\ell}-x_{**}} - \sqrt{x - x_{**}} \right).$$ Since $\beta(x_{\ell})=0$, we have $$\lim_{x \to x_{**} }\beta(x) \leq 2 \left( m\sqrt{x_{\ell} - x_{**} } + \sqrt{3M}  \right) \left( \sqrt{x_{\ell}-x_{**}} \right),$$ which is a contradiction.
\end{proof}

Now that we know $\beta$ is bounded from above, we can show that $x_{**} > 0$ when there exists $x_m \in (x_{**},x_*)$ so that $\beta'(x_m) = 0$.
\begin{claim}
\label{c3:sm:b}
Suppose there exists $x_m \in (x_{**},x_*)$ so that $\beta'(x_m) = 0$. Then $x_{**} > 0$.
\end{claim}
\begin{proof}
If there exists $x_m \in (x_{**},x_*)$ so that $\beta'(x_m) = 0$, then $\beta''>0$ and there exists $\epsilon>0$ and $x_{\epsilon} \in (x_{**},x_m)$ so that $\frac{\beta'(x_{\epsilon})}{\sqrt{1 + \beta'(x_{\epsilon})^2}} = -\epsilon$. Also, by Lemma~\ref{beta:bd}, there exists $M \geq 0$ so that $\beta < M$.

Let $\theta(x) = \arctan \beta'(x)$. Then $$\frac{d}{dx} \left( \log \sin \theta(x) \right) = \frac{1}{2}x - \frac{1}{x} - \frac{\beta(x)}{2 \beta'(x)} \leq \frac{1}{2}x - \frac{1}{x} - \frac{M}{2 \beta'(x_\varepsilon)},$$ for $x \in (x_{**},x_\varepsilon)$. Integrating the inequality from $x$ to $x_{\epsilon}$: $$ \log \left( \frac{\sin \theta(x_\varepsilon)}{\sin \theta(x)} \right) \leq \frac{1}{4}(x_\varepsilon)^2 + \log \left( \frac{x}{x_\varepsilon} \right) + \frac{M x_\varepsilon}{2(-\beta'(x_\varepsilon))}.$$ Since $\sin \theta(x_\varepsilon) = \frac{\beta'(x_{\epsilon})}{\sqrt{1 + \beta'(x_{\epsilon})^2}} = -\epsilon$ and $\sin \theta(x_{**}) = -1$, we have $$\varepsilon \leq \left( \frac{x_{**}}{x_\varepsilon} \right) e^{\frac{1}{4}(x_\varepsilon)^2 + \frac{M x_\varepsilon}{2(-\beta'(x_\varepsilon))}},$$ which proves the claim.
\end{proof}

\begin{remark}
At this point, we can give a basic description of the $\beta$ curves: For $b>0$, let $(x_*^b,z_*^b)$ be the blow-up point of $\gamma_b$, and let $\beta_b$ be the unique solution of~(\ref{gamma:ode}) with $\beta_b(x_*^b) = z_*^b$ and $\lim_{x \to x_*^b} \beta_b'(x) = \infty$. Then there exists $x_{**}^b \in [0,\sqrt{2})$ so that $\beta_b$ is defined on $(x_{**}^b,x_*^b)$ and either $\beta_b$ blows-up as $x \to x_{**}^b$ or $x_{**}^b=0$. Also, $\beta_b$ is bounded from above, and if there exists $x_m^b$ for which $\beta_b'(x_m^b) = 0$, then $\beta_b''>0$ on $(x_{**}^b, x_*^b)$ and $x_{**}^b>0$.
\end{remark}

\subsection{A note on the blow-up of $\beta$ at $x_{**}$}

Now that we know the basic shape of $\beta$ we discuss the dependence of $\beta$ and $(x_{**},z_{**})$ on the initial height. As in Section~\ref{cfsb}, the following proposition is a consequence of the existence, uniqueness, and continuity theory for differential equations.

\begin{proposition}
\label{prop:beta:cont}
For $b>0$, let $\gamma_b$ denote the solution of~\emph{(\ref{gamma:ode})} with $\gamma_b(0) = b$ and $\gamma_b'(0) = 0$. Let $x^{b}_*$ denote the point where $\gamma_{b}$ blows-up, and let $z^b_* = \gamma_b(x^b_*)$.  Let $\beta_b$ denote the unique solution of~\emph{(\ref{gamma:ode})} with $\beta_b(x^b_*) = z_*^b$ and $\lim_{x \to x^b_*} \beta'(x) = \infty$, and let $x_{**}^b \in [0,x_*^b)$ be chosen so that $\beta_b$ is smooth on $(x_{**}^b, x_*^b)$ and either $\beta_b$ blows-up at $x_{**}^b$ or $x_{**}^b=0$

Fix $b_0>0$. Then, for $\varepsilon > 0$, there exists $\delta>0$ so that if $|b-b_0| < \delta$ and $b>0$, then $\beta_b(x)$ is defined on $[x^{b_0}_{**} + \varepsilon, x^{b_0}_* - \varepsilon]$. Moreover, $$|\beta_b(x) - \beta_{b_0}(x)| + |\beta_b'(x) - \beta_{b_0}'(x)| < \varepsilon,$$ for $x \in [x^{b_0}_{**} + \varepsilon, x^{b_0}_* - \varepsilon]$.
\end{proposition}

Applying Proposition~\ref{prop:beta:cont}, we have the following continuity result.

\begin{proposition}
\label{beta:blowup:pt}
Fix $b_0 > 0$. Suppose there exists $x_m^{b_0} \in (x_{**}^{b_0},x_*^{b_0})$ so that $\beta_{b_0}'(x_m^{b_0}) = 0$ and hence $x_{**}^{b_0} >0$. Suppose $\beta_{b_0}(x_{**}^{b_0}) = z_{**}^{b_0}$, where $|z_{**}^{b_0}| < \infty$. Then, for $\varepsilon > 0$, there exists $\delta>0$, so that if $|b-b_0| < \delta$ and $b >0$, then the solution $\beta_b$ blows-up at the point $x_{**}^b>0$ with $\beta_b(x_{**}^b) = z_{**}^b$ and $|z_{**}^b| < \infty$. Furthermore, if $|b-b_0| < \delta$ and $b>0$, then $|x^b_{**} - x^{b_0}_{**}| < \varepsilon$ and $|z^b_{**} - z^{b_0}_{**}| < \varepsilon$.
\end{proposition}

\begin{proof}
Since $x_{**}^{b_0} >0$ and $|z_{**}^{b_0}| < \infty$, we can continue $(x,\beta_{b_0}(x))$ past the blow-up point $(x_{**}^{b_0},z_{**}^{b_0})$ along a curve $(\bar{\alpha}_{b_0}(z),z))$ just as we did for $(x,\gamma(x))$ at the blow-up point $(x_*,z_*)$. Using Proposition~\ref{prop:beta:cont} we can show that the $\bar{\alpha}$ curves depend continuously on the initial height. In particular for $b$ in a neighborhood of $b_0$, the blow-up points $(x_{**}^b,z_{**}^b)$ will exist and depend continuously on $b$.
\end{proof}

\subsection{Behavior of $\beta$ for small $b>0$}

Fix $b \in (0, \bar{b}]$, and let $\gamma$ be the solution of~(\ref{gamma:ode}) with $\gamma(0) = b$ and $\gamma'(0) = 0$. Let $x_*$ denote the point where $\gamma$ blows-up, and let $z_* = \gamma(x_*)$. Also, let $\beta$ denote the unique solution of~(\ref{gamma:ode}) with $\beta(x_*) = z_*$ and $\lim_{x \to x_*} \beta'(x) = \infty$. We know there is a point $x_{**} \in [0,\sqrt{2})$ so that $\beta$ is defined on $(x_{**},x_*)$ and either blows-up as $x \to x_{**}$ or $x_{**}=0$. We also know that $$x_*^b \geq \sqrt{ \log{\frac{2}{\pi b^2}} }$$ and $$\frac{-12}{ \sqrt{\log{\frac{2}{\pi b^2}}} } \leq z_*^b < 0.$$

\begin{claim}
\label{c1:sm:b}
Suppose $x_* \geq 4$. Then there exists a point $x_m \in [x_* - 2, x_*)$ so that $\beta'(x_m) = 0$.
\end{claim}
\begin{proof}
Suppose $\beta'(x) > 0$ for $x \in [x_*-2,x_*)$. When $x_* \geq 4$, we have $x_* - 2 \geq \sqrt{2}$ and $\beta(x_*) < 0$. Using~(\ref{gamma:ode}), we see that $\beta'' > 0$ in $[x_*-2,x_*)$ and $\beta''(x_*-2) \geq -\frac{1}{2} \beta(x_*-2)$. Then, using~(\ref{gamma:eq:3}), we have $\beta''' > 0$ in $[x_*-2,x_*)$. Integrating from $x_*-2$ to $x$, we get $$\beta(x) \geq \beta(x_*-2) \left[ 1 - \frac{1}{4} \left( x - (x_*-2) \right)^2 \right],$$ for $x \in [x_*-2,x_*)$. This tells us that $\beta(x_*) \geq 0$, which is a contradiction.
\end{proof}

Let $\bar{b} >0$ be given as in the conclusion of Proposition~\ref{gamma:small:height} so that if $b \in (0,\bar{b}]$, then $\gamma \leq 0$ for $x\in[2\sqrt{2}, x_*)$. We also assume that $\bar{b}$ is chosen so small that $x_m > 2\sqrt{2}$ and $z_* \geq -\frac{1}{8} \geq -\frac{3\sqrt{2}}{4}$ when $b \in (0, \bar{b}]$. (It suffices to choose $\bar{b} < \frac{\sqrt{2}}{\sqrt{\pi} e^{4608}}$.)

\begin{lemma}
If $b \in (0,\bar{b}]$, then $2 z_{*} \leq \beta(x) < 0$ for $x \in [2 \sqrt{2}, x_*]$.
\end{lemma}
\begin{proof}
Let $\alpha(z)$ denote the curve that connects $\gamma$ and $\beta$. Then $\alpha$ is a solution of~(\ref{alpha:eq}) with $\alpha(z_*) = x_*$ and $\alpha'(z_*) = 0$. Using~(\ref{alpha:eq}) we have $\alpha''(z_*) = \frac{1}{x_*} - \frac{1}{2} x_*$ and $\alpha'''(z_*) = \frac{1}{2}z_*(\frac{1}{x_*} - \frac{1}{2} x_*)$ so that $$\alpha(z) = x_* + \frac{1}{2} \left( \frac{1}{x_*}\ - \frac{1}{2}x_* \right) (z-z_*)^2 + \frac{1}{12} z_* \left( \frac{1}{x_*}\ - \frac{1}{2}x_* \right) (z-z_*)^3 + \mathcal{O}(|z-z_*|^4)$$ as $z \to z_*$. Since $x_* > \sqrt{2}$ and $z_* < 0$, the coefficient of the $(z-z_*)^3$ term is positive. Now, for $x<x_*$ and near $x_*$, we can find $s, \, t >0$ so that $$\alpha(z_* + t ) = x = \alpha(z_* -s),$$ and it follows from the previous formula for $\alpha(z)$ that $t>s$. 

To prove the lemma, we consider the function $\delta(x) = \gamma(x) + \beta(x)$. For $x \in [2\sqrt{2} , x_*)$, since $\gamma(x) \leq 0$ and $\beta(x) < 0$, we have $$\delta(x)  < 0.$$  Also, by the previous paragraph, $$\delta(x) = (z_* + t) + (z_* -s)  > 2 z_* = \delta(x_*),$$ for $x<x_*$ and near $x_*$. 

Now we are in position to show that $\delta > 2z_*$ on $[2\sqrt{2}, x_*)$. Suppose $\delta(x) = 2z_*$ for some $x \in [2\sqrt{2}, x_*)$. It follows from the previous discussion that $\delta$ achieves a negative maximum at some point $\bar{x} \in (2\sqrt{2}, x_*)$. At this point we have $\delta''(\bar{x}) \leq 0$ and $$\frac{\delta''(\bar{x})}{1 + \gamma'(\bar{x})^2} = -\frac{1}{2} \delta(\bar{x}) > 0,$$ which is a contradiction. Therefore, $\delta > 2z_*$ on $[2\sqrt{2}, x_*)$. Since $\gamma(x) \leq 0$ on $[2\sqrt{2}, x_*)$, this completes the proof of the lemma.
\end{proof}

\begin{claim}
\label{c2:sm:b}
Let $b \in (0, \bar{b}]$. If $\beta<0$ on $(x_{**},x_*)$, then $$x_{**} \leq \frac{8}{\pi - \sqrt{2}}(-z_*).$$
\end{claim}
\begin{proof}
By our assumptions on $\bar{b}$, we know that $x_m> 2\sqrt{2}$, $\beta''>0$, and $\beta(2\sqrt{2}) \geq 2z_*$. Using equation~(\ref{gamma:ode}), we get the estimate $\beta'(2\sqrt{2}) > \frac{\sqrt{2}}{3} \beta(2\sqrt{2}) \geq \frac{2\sqrt{2}}{3}z_*$ so that $\beta'(2\sqrt{2}) > -1$. For $x \in (x_{**}, \sqrt{2})$, we have $$\frac{d}{dx} \left( \arctan \beta'(x) \right) = \left( \frac{1}{2}x - \frac{1}{x} \right) \beta'(x) - \frac{1}{2} \beta(x) \leq -\frac{1}{x_{**}} \beta'(x) - \frac{1}{2} \beta(2\sqrt{2}).$$ Integrate from $x_{**}$ to $2\sqrt{2}$, $$\arctan \beta'(2\sqrt{2}) + \frac{\pi}{2} \leq \left( \frac{1}{x_{**}} + \sqrt{2} \right) (-\beta(2\sqrt{2})).$$ Since $\beta'(2\sqrt{2}) > -1$,  and $\beta(2\sqrt{2}) \geq 2z_*$, this becomes $$\frac{\pi}{4} \leq \left( \frac{1}{x_{**}} + \sqrt{2} \right) 2 (-z_*).$$ Rearranging this inequality to estimate $x_{**}$ and using $z_* \geq -1/8$, we have $$x_{**} \leq \frac{8}{\pi - \sqrt{2}}(-z_*).$$
\end{proof}

\begin{proposition}
\label{beta:crossing}
There exists $\tilde{b}>0$ so that for $b \in (0, \tilde{b}]$ there is a point $x_m^b \in (x_{**}^b,x_*^b)$ so that $\beta_b'(x_m^b) =0$ and $0 < \beta_b(x_{**}^b) < \infty$.
\end{proposition}
\begin{proof}
Fix $b \in (0, \bar{b}]$, and let $\beta = \beta_b$. By Claim~\ref{c1:sm:b} we know there exists $x_m > 2\sqrt{2}$ so that $\beta'(x_m) = 0$. It follows that $\beta''>0$ on $(x_{**},x_*)$ and $\beta' < 0$ on $(x_{**}, x_m)$. From Lemma~\ref{beta:bd}, we know that $\beta(x_{**}) < \infty$.

Suppose $\beta(x_{**}) \leq 0$ so that $\beta<0$ in $(x_{**},x_m)$. By Claim~\ref{c2:sm:b} we know that $x_{**}<1$ for $b$ sufficiently small. Then, using the equation~(\ref{gamma:eq:3}) for $\beta'''$, we see that $\beta'''<0$ in $(x_{**}, \sqrt{2}]$. We know that $\beta''(\sqrt{2}) \geq -\frac{1}{2} \beta(\sqrt{2})$, and thus $\beta''(x) \geq -\frac{1}{2} \beta(\sqrt{2})$, for $x \in (x_{**}, \sqrt{2}]$. Integrating from $1$ to $\sqrt{2}$ and using that $\beta$ is decreasing on $(1,\sqrt{2})$, we have  $$\beta'(1) \leq \frac{\sqrt{2} - 1}{2} \beta(1).$$

Let $x \in (x_{**}, 1)$. Under the above assumptions, we may write~(\ref{gamma:ode}) as $\frac{\beta''(x)}{\beta'(x)} \leq \left(\frac{1}{2}x - \frac{1}{x} \right)$, and integrating this inequality from $x$ to $1$, we get $\beta'(x) \leq \frac{\beta'(1)}{x} e^{-\frac{1}{4}}$. Integrating again, $$\beta(x) \geq \beta(1) + \beta'(1) e^{-\frac{1}{4}} \log{x}.$$ Combining this with $\beta'(1) \leq \frac{\sqrt{2} - 1}{2} \beta(1)$, we see that $$0>\beta(x) \geq \beta'(1) \left( \frac{2}{\sqrt{2}-1} + e^{-\frac{1}{4}} \log{x} \right).$$ If $b$ is chosen sufficiently small (so that $x_{**} < e^{-\frac{2e^{1/4}}{\sqrt{2}-1}}$), then we have a contradiction. Therefore, $0 < \beta_b(x_{**}^b) < \infty$ when $b>0$ is sufficiently small.
\end{proof}


\section{An Immersed $S^2$ Self-Shrinker}
\label{proof}

In this section, we complete the proof of Theorem~\ref{thm:1}. We consider the set $$\{ \tilde{b} \, : \, \forall b \in (0, \tilde{b}], \exists x_m^b \in (x_{**}^b, x_*^b) \textrm{ so that } \beta_b'(x_m^b) = 0 \textrm{ and } \beta_b(x_{**}^b) > 0\}.$$ By Proposition~\ref{beta:crossing}, we know that this set is non-empty. Following Angenent's argument in~\cite{A}, we let $b_0$ be the supremum of this set: $$b_0 = \sup \{ \tilde{b} \, : \, \forall b \in (0, \tilde{b}], \exists x_m^b \in (x_{**}^b, x_*^b) \textrm{ so that } \beta_b'(x_m^b) = 0 \textrm{ and } \beta_b(x_{**}^b) > 0\}.$$ Since $\beta_b(x) = - \sqrt{4 - x^2}$ when $b=2$, we know that $b_0 \leq 2$. We want to show $\beta_{b_0}$ intersects the $x$-axis perpendicularly at $x_{**}^{b_0}$.

\begin{claim}
\label{c1:im}
$b_0 < 2$.
\end{claim}
\begin{proof}
We will prove it is impossible to have $\beta_{b_0}'>0$ and $\beta''_{b_0} \geq 0$ in $(x_{**}^{b_0},x_*^{b_0})$. In particular, this will show $b_0 \neq 2$.

Suppose $\beta_{b_0}'>0$ and $\beta''_{b_0} \geq 0$ in $(x_{**}^{b_0},x_*^{b_0})$. Then $x_{**}^{b_0} = 0$, and there exists $m>0$ so that $\beta_{b_0}(x) \leq -m$ for $x \in (0,1]$. Let $b_n$ be an increasing sequence that converges to $b_0$, and let $\beta_n$ be the solution of~(\ref{gamma:ode}) corresponding to the initial height $b_n$. Fix $\varepsilon > 0$. By Proposition~\ref{prop:beta:cont}, there exists $N=N(\varepsilon)>0$ so that for $n>N$, we have $x_{**}^{n} < \varepsilon$ and $\beta_n(\varepsilon) < -m$. We know that $\beta_n$ intersects the $x$-axis at some point $x_{\ell}^n < \varepsilon$, and we see that the curve $(x,\beta_n(x))$ may be written as the curve $(\bar{\alpha}_n(z),z)$ for $z \in [-m/2,0]$, where $\bar{\alpha}_n$ is a solution of~(\ref{alpha:eq}). In fact, we have the estimates $$0 < \bar{\alpha}_n(z) < \varepsilon$$ and $$-\frac{2 \varepsilon}{m} \leq \bar{\alpha}_n'(z) \leq 0.$$ Using~(\ref{alpha:eq}), we get an estimate for $\bar{\alpha}_n''(z)$ when $z \in [-m/2,0]$: $$\bar{\alpha}_n''(z) \geq \frac{1}{\varepsilon} - \frac{1}{2} \varepsilon \geq \frac{1}{2\varepsilon},$$ for small $\varepsilon>0$. Integrating repeatedly from $z$ to $0$: $$\bar{\alpha}_n(z) \geq \bar{\alpha}_n(0) + \bar{\alpha}_n'(0)z + \frac{1}{4\varepsilon} z^2 \geq \frac{1}{4\varepsilon} z^2.$$ Taking $z=-m/2$, we have $\bar{\alpha}_n(-m/2) \geq \frac{m^2}{16 \varepsilon}$. This implies that the point $x' \in (x_{**}^n, \varepsilon)$ for which $\beta_n(x') = -m/2$ satisfies $x' \geq \frac{m^2}{16 \varepsilon}$. Therefore, $\frac{m^2}{16 \varepsilon} \leq x' < \varepsilon$, which is a contradiction when $\varepsilon > 0$ is sufficiently small.

\end{proof}

\begin{claim}
\label{c2:im}
There exists $x^{b_0}_m \in (x^{b_0}_{**}, x^{b_0}_*)$ so that $\beta_{b_0}'(x^{b_0}_m) = 0$.
\end{claim}
\begin{proof}
Suppose not. Then $\beta_{b_0}' > 0$ in $(x^{b_0}_{**}, x^{b_0}_*)$. By the proof of Claim~\ref{c1:im} we know that it is impossible to have $\beta_{b_0}'>0$ and $\beta''_{b_0} \geq 0$ in $(x_{**}^{b_0},x_*^{b_0})$. Therefore, we must have $\beta_{b_0}'' < 0$ at some point. Let $b_n$ be an increasing sequence that converges to $b_0$, and let $\beta_n$ be the solution of~(\ref{gamma:ode}) corresponding to the initial height $b_n$. Since each curve $\beta_n$ crosses the $x$-axis we know from Claim~\ref{beta:claim1} that $\beta_n'' >0$. Since $\beta_{b_0}'' < 0$ at some point $x'$, we can use Proposition~\ref{prop:beta:cont} to find an $N>0$ so that $\beta_n''(x') < 0$ for $n>N$, which contradicts the fact that $\beta_n'' >0$.
\end{proof}

Since there exists $x^{b_0}_m \in (x^{b_0}_{**}, x^{b_0}_*)$ with $\beta_{b_0}'(x^{b_0}_m) = 0$, we know that $\beta_{b_0}''>0$ on $(x_{**}^{b_0}, x_*^{b_0})$ and $x_{**}^{b_0} > 0$. We also know that there is a point $z_{**}^{b_0}$ with $|z_{**}^{b_0}| < \infty$ so that $\beta_{b_0}(x_{**}^{b_0}) = z_{**}^{b_0}$.

\begin{claim}
\label{beta:intersect}
$z_{**}^{b_0} = 0$.
\end{claim}
\begin{proof}
We know that $x_{**}^{b_0} > 0$ and $|z_{**}^{b_0}| < \infty$. We also know that there exists $x^{b_0}_m \in (x^{b_0}_{**}, x^{b_0}_*)$ so that $\beta_{b_0}'(x^{b_0}_m) = 0$. It follows from Proposition~\ref{prop:beta:cont} (and Claim~\ref{beta:claim1}) that there is a $\delta'>0$ so that when $|b-b_0|<\delta'$ there exists $x_m^b \in (x_{**}^b,x_*^b)$ with $\beta_b'(x_m^b) = 0$. 

If $z_{**}^{b_0} > 0$, then applying Proposition~\ref{beta:blowup:pt}, we can find $\delta \in (0, \delta')$ so that the curve $\beta_b$ has a finite blow-up point $(x_{**}^b,z_{**}^b)$ with $z_{**}^b>0$ when $|b-b_0|< \delta$ . Then $b_0 + \delta/2 \in \{ \tilde{b} \, : \, \forall b \in (0, \tilde{b}], \exists x_m^b \in (x_{**}^b, x_*^b) \textrm{ so that } \beta_b'(x_m^b) = 0 \textrm{ and } \beta_b(x_{**}^b) > 0 \}$, which contradicts the definition of $b_0$. On the otherhand, if $z_{**}^{b_0} < 0$, then Proposition~\ref{beta:blowup:pt} tells us there exists $\delta \in (0, \delta')$ so that the curve $\beta_b$ has a finite blow-up point $(x_{**}^b,z_{**}^b)$ with $z_{**}^b<0$ when $|b-b_0|< \delta$, but this also contradicts the definition of $b_0$. Therefore, $z_{**}^{b_0} = 0$.
\end{proof}

It follows that the curve $(x,\beta_{b_0}(x))$ intersects the $x$-axis perpendicularly at the point $(x_{**}^{b_0}, 0)$, where $x_{**}^{b_0} \in (0,\sqrt{2})$. Now we can describe the curve $\mathcal{C}$ in the $(x,z)$-plane whose rotation about the $z$-axis is an immersed $S^2$ self-shrinker.

\begin{proof}[Proof of Theorem~\ref{thm:1}]
Let $\mathcal{C}$ be the curve in the $(x,z)$-plane obtained by following along $(x,\gamma_{b_0}(x))$ as $x$ goes from $0$ to $x_*^{b_0}$, following along $(x,\beta_{b_0}(x))$ as $x$ goes from $x_*^{b_0}$ to $x_{**}^{b_0}$, and then following the reflections $(x,-\beta_{b_0}(x))$ and $(x,-\gamma_{b_0}(x))$. That is, $\mathcal{C} = \gamma_{b_0} \cup \beta_{b_0} \cup -\beta_{b_0} \cup -\gamma_{b_0}$ (see Figure~\ref{fig_immersed}). The curve $\mathcal{C}$ intersects the $z$-axis perpendicularly at the two points $(0, b_0)$ and $(0, -b_0)$, and the rotation of $\mathcal{C}$ about the $z$-axis is smooth in a neighborhood of these points. In addition, at the vertical tangent points where the $\gamma$ and $\beta$ curves meet, $\mathcal{C}$ can be represented as an $\alpha$ curve, and we see that $\mathcal{C}$ is a smooth curve, whose rotation $M$ about the $z$-axis is a smooth manifold. In fact, $M$ is the image of a smooth immersion from $S^2$ into $\RR^3$. By construction, the $\gamma$, $\alpha$, and $\beta$ curves are solutions of the differential equation that corresponds to self-shrinkers with rotational symmetry. Also, the $\gamma_{b_0}$ and $-\beta_{b_0}$ curves intersect transversally. Therefore, the surface $M$ is an immersed and non-embedded $S^2$ self-shrinker in $\RR^3$.
\end{proof}


\section{Appendix: Existence of Solutions Near $x=0$}

In this section, we use power series to construct solutions to the differential equation
\begin{equation}
\label{gamma:ode:appendix}
\frac{\gamma''}{1+(\gamma')^2} = \left(\frac{1}{2}x - \frac{1}{x} \right) \gamma' -\frac{1}{2} \gamma,
\end{equation}
when $\gamma(0) \in \RR$ and $\gamma'(0) = 0$.

We look for solutions of the form: $$\gamma(x) = \sum_{i=0}^{\infty} a_i x^i.$$ If we assume that we can differentiate the power series term by term so that $$\gamma'(x) = \sum_{i=0}^{\infty} (i+1) a_{i+1} x^i$$ and $$\gamma''(x) = \sum_{i=0}^{\infty} (i+2)(i+1) a_{i+2} x^i,$$ then the condition that $\gamma$ satisfies~(\ref{gamma:ode:appendix}): $$\gamma'' = -\frac{1}{2} \gamma + \frac{1}{2} x \gamma' - \frac{1}{x} \gamma' - \frac{1}{2} \gamma (\gamma')^2 + \frac{1}{2} x (\gamma')^3 - \frac{1}{x} (\gamma')^3,$$ is a condition on the coefficients $\{a_i\}$. Namely, $a_0 = \gamma(0)$, $a_1=0$, and
\begin{eqnarray}
(m+2)(m+1)a_{m+2} & = & -\frac{1}{2} a_m + \frac{1}{2} m a_m - (m+2) a_{m+2} \nonumber \\
& &- \frac{1}{2} \sum_{i+j+k = m}(i+1)(j+1) a_{i+1} a_{j+1} a_k \nonumber \\
& & + \frac{1}{2} \sum_{i+j+k = m -1}(i+1)(j+1)(k+1) a_{i+1} a_{j+1} a_{k+1} \nonumber \\
& & - \sum_{i+j+k = m + 1} (i+1)(j+1)(k+1) a_{i+1} a_{j+1} a_{k+1}. \nonumber
\end{eqnarray}
The previous equation simplifies to:
\begin{eqnarray}
\label{coef:formula}
(m+2)^2 a_{m+2} & = & \frac{1}{2}(m-1) a_m - \frac{1}{2} a_0 \sum_{i+j = m}(i+1)(j+1) a_{i+1} a_{j+1} \\
& & + \frac{1}{2}  \sum_{i+j+k=m-1} (i+1)(j+1)k \cdot a_{i+1} a_{j+1} a_{k+1} \nonumber \\
& & - \sum_{i+j+k = m + 1} (i+1)(j+1)(k+1) a_{i+1} a_{j+1} a_{k+1}. \nonumber
\end{eqnarray}

\begin{claim}
$a_{2i+1}=0$
\end{claim}
\begin{proof}
This follows from the above formula for $a_m$ and induction.  We know $a_1 = 0$. Suppose $a_{2i+1} = 0$ for all $i<m$ and consider $a_{2m+1}$. Every term in the expression for $a_{2m+1}$ contains a term of the form $a_{2i+1}$, and thus $a_{2m+1}=0$.
\end{proof}

In order to construct a solution of~(\ref{gamma:ode:appendix}) we need $\sum a_i x^i$ to be a convergent power series, and hence we need an estimate on the coefficients $a_{2m}$. 

\begin{claim}
For each $M>0$, there exists $A=A(M)>0$ so that if $|a_0| \leq M$, then $$|a_{2m}| \leq \frac{A^{2m-1}}{(2m)^3}.$$
\end{claim}
\begin{proof}
Fix $M>0$, and use~(\ref{coef:formula}) to choose $A>0$ so that the estimate holds for $m=1,2,3$. Arguing inductively, suppose $a_{2i} \leq \frac{A^{2i-1}}{(2i)^3}$ when $i \leq m$, and consider $a_{2m+2}$. From~(\ref{coef:formula}), we have
\begin{eqnarray}
\label{ps:even}
(2m+2)^2 a_{2m+2} & = & \frac{1}{2}(2m-1) a_{2m} \\
& & - \frac{1}{2} a_0 \sum_{ \substack{i+j = 2m \\ i,j \textrm{ odd} } }(i+1)(j+1) a_{i+1} a_{j+1} \nonumber \\
& & + \frac{1}{2}  \sum_{ \substack{i+j+k=2m-1 \\ i,j,k \textrm{ odd} } } (i+1)(j+1)k \cdot a_{i+1} a_{j+1} a_{k+1} \nonumber \\
& & - \sum_{ \substack{i+j+k = 2m + 1 \\ i,j,k \textrm{ odd} } } (i+1)(j+1)(k+1) a_{i+1} a_{j+1} a_{k+1}. \nonumber
\end{eqnarray}

To estimate the terms with sums, we need the following inequality: 
\begin{equation}
\label{ps:in1}\sum_{ \substack{i+j = 2N \\ i,j \textrm{ odd} } } \frac{1}{(i+1)^2} \frac{1}{(j+1)^2} \leq \frac{2}{(2N+2)^2}.
\end{equation}
This inequality follows from the identity: $$\sum_{ \substack{i+j = 2N \\ i,j \textrm{ odd} } } \frac{1}{(i+1)^2} \frac{1}{(j+1)^2} = \frac{2}{(2N+2)^2} \sum_{ \substack{i=1 \\ i \textrm{ odd}} }^{2N-1} \frac{1}{(i+1)^2} + \frac{4}{(2N+2)^3} \sum_{ \substack{i=1 \\ i \textrm{ odd}} }^{2N-1} \frac{1}{i+1}.$$ Applying~(\ref{ps:in1}) twice, we have the following inequality:
\begin{equation}
\label{ps:in2}
\sum_{ \substack{i+j+k = 2N-1 \\ i,j,k \textrm{ odd} } } \frac{1}{(i+1)^2} \frac{1}{(j+1)^2} \frac{1}{(k+1)^2} \leq \frac{4}{(2N+2)^2}.
\end{equation}

Now, we can use~(\ref{ps:in1}),~(\ref{ps:in2}), and the inductive hypothesis to estimate each term on the right hand side of formula~(\ref{ps:even}):
\begin{equation}
\frac{1}{2}(2m-1)|a_{2m}| \leq \frac{1}{2} \frac{2m-1}{(2m)^3} A^{2m-1} \leq \left[ \frac{m+1}{(2m)^2} \right] \frac{1}{2m+2} A^{2m+1}, \nonumber
\end{equation}
\begin{equation}
\frac{1}{2} \left| a_0 \sum_{ \substack{i+j = 2m \\ i,j \textrm{ odd} } }(i+1)(j+1) a_{i+1} a_{j+1} \right| \leq \left[ \frac{1}{2m+2} \right] \frac{1}{2m+2} A^{2m+1}, \nonumber
\end{equation}
\begin{equation}
\frac{1}{2}  \left| \sum_{ \substack{i+j+k=2m-1 \\ i,j,k \textrm{ odd} } } (i+1)(j+1)k \cdot a_{i+1} a_{j+1} a_{k+1} \right| \leq \left[ \frac{1}{m+1} \right] \frac{1}{2m+2} A^{2m+1}, \nonumber
\end{equation}
\begin{equation}
\left| \sum_{ \substack{i+j+k = 2m + 1 \\ i,j,k \textrm{ odd} } } (i+1)(j+1)(k+1) a_{i+1} a_{j+1} a_{k+1} \right| \leq \left[ \frac{2}{m+2} \right] \frac{1}{2m+2} A^{2m+1}. \nonumber
\end{equation}
Applying these estimates to~(\ref{ps:even}) when $m \geq 3$, we see that
\begin{eqnarray}
(2m+2)^2 |a_{2m+2}| & \leq & \left[\frac{m+1}{(2m)^2} + \frac{1}{2m+2} + \frac{1}{m+1} + \frac{2}{m+2} \right] \frac{1}{2m+2} A^{2m+1} \nonumber \\
& \leq & \frac{1}{2m+2} A^{2m+1}. \nonumber
\end{eqnarray}
\end{proof}

The estimates on the coefficients $a_i$ imply that the power series $\gamma(x) = \sum a_i x^i$ is an analytic function on $[0,1/A]$. By the previous discussion, $\gamma(x)$ is the unique analytic solution of~(\ref{gamma:ode:appendix}) in $(0,1/A]$. Furthermore, since the coefficients $a_i$ depend continuously on $a_0$, the solution $\gamma(x)$ depends continuously on the initial height $\gamma(0)$.


\bibliographystyle{amsplain}

\end{document}